\documentclass[reqno]{amsart}
\usepackage{amssymb}
\usepackage[usenames, dvipsnames]{color}
\usepackage{hyperref}

\usepackage{verbatim}

\numberwithin{equation}{section}

\newtheorem{theorem}{Theorem}[section]

\newtheorem{lemma}[theorem]{Lemma}

\newtheorem{statement}[theorem]{Statement}

\theoremstyle{definition}
\newtheorem{remark}[theorem]{Remark}

\theoremstyle{definition}

\theoremstyle{definition}

\makeatletter
\def\dashint{\operatorname%
{\,\,\text{\bf--}\kern-.98em\DOTSI\intop\ilimits@\!\!}}
\makeatother

\def\bR{\mathbb{R}}

\def\fL{\mathfrak{L}}

\def\cL{\mathcal{L}}

\begin{document}
\title[Elliptic equations with piecewise constant coefficients]{On the impossibility of $W_p^2$ estimates for elliptic equations with piecewise constant coefficients}

\author[H. Dong]{Hongjie Dong}
\address[H. Dong]{Division of Applied Mathematics, Brown University, 182 George Street, Providence, RI 02912, USA}
\email{Hongjie\_Dong@brown.edu}
\thanks{H. Dong was partially supported by NSF grant number DMS-1056737.}

\author[D. Kim]{Doyoon Kim}
\address[D. Kim]{Department of Applied Mathematics, Kyung Hee University, 1732 Deogyeong-daero, Giheung-gu, Yongin-si, Gyeonggi-do 446-701, Republic of Korea}
\email{doyoonkim@khu.ac.kr}
\thanks{D. Kim was supported by Basic Science Research Program through the National Research Foundation of Korea (NRF) funded by the Ministry of Science, ICT \& Future Planning (2011-0013960).}



\begin{abstract}
In this paper, we present counterexamples showing that for any $p\in (1,\infty)$, $p\neq 2$, there is a non-divergence form uniformly elliptic operator with piecewise constant coefficients  in $\bR^2$ (constant on each quadrant in $\bR^2$) for which there is no $W^2_p$ estimate. The corresponding examples in the divergence case are also discussed.
One implication of these examples is that the ranges of $p$ are sharp in the recent results obtained in \cite{MR3033629,DongKim2013} for non-divergence type elliptic and parabolic equations in a half space with the Dirichlet or Neumann boundary condition when the coefficients do not have any regularity in a tangential direction.
\end{abstract}

\maketitle

\section{Introduction and main results}
We consider elliptic operators in non-divergence form
$$
\cL u = a^{ij}D_{ij} u,
$$
where
\begin{equation}
							\label{eq0217_1}
\delta |\xi|^2 \le a^{ij} \xi_i \xi_j,
\quad
|a^{ij}| \le \delta^{-1},
\quad
\delta \in (0,1].
\end{equation}

The $L_p$-theory of second-order elliptic and
parabolic equations with discontinuous coefficients has been studied extensively in the last fifty years. In the special case when the dimension $d=2$, it is well known that the $W^2_2$ estimate holds for uniformly elliptic operators with general bounded and measurable coefficients. See, for instance, \cite{MR1511579, MR0204845}.
On the other hand, a celebrated counterexample in \cite{MR0201816} and \cite{MR2260015} indicates that when $d \ge 3$ in general there is no $W^2_2$ estimate for elliptic operators with bounded measurable coefficients even if they are discontinuous only at a single point. Another example due to Ural'tseva \cite{MR0226179} (see also \cite{MR2667637}) shows the impossibility of the $W^2_p$ estimate when $d\ge 2$ and $p\neq 2$.
We note that in Ural'tseva's example, the coefficients are continuous except at a single point ($d=2$) or a line ($d=3$). In \cite{MR1612401}, Nadirashvili showed that the weak uniqueness for martingale problems may fail if coefficients are merely measurable and $d \ge 3$. These examples imply that in general there does not exist a solvability theory for uniformly elliptic operators with bounded and measurable coefficients. Thus many efforts have been made to treat particular types of discontinuous coefficients.

In \cite{MR0224996} Campanato extended the aforementioned result in \cite{MR1511579, MR0204845} to the case when $d=2$ and $p$ is in a neighborhood of $2$, the size of which depends on the ellipticity constant $\delta$. A corresponding result for parabolic equations can be found in \cite{MR0255954}.
By using explicit representation formulae, Lorenzi \cite{MR0296490, MR0382836} studied the $W^2_2$ and $W^2_p$, $1<p<\infty$, estimates for elliptic equations in $\bR^d$ with coefficients which are constant on each half space.
See \cite{MR0499720, MR2240183} for similar results for parabolic equations, and \cite{MR2276531} for elliptic equations in $\bR^d$ with leading coefficients discontinuous at finitely many parallel hyperplanes.
We also refer the reader to \cite{MR2338417, MR2678983, MR2667637, MR3033629, DongKim2013} and the references therein for some recent developments for equations with coefficients only measurable in some directions. In particular, it is proved in \cite{MR3033629} that the $W^2_p$ estimate holds for elliptic equations in a half space with the zero Dirichlet (or Neumann) boundary condition when coefficients are only measurable in a tangential direction to the boundary and $p\in (1,2]$ (or $p\in [2,\infty)$, respectively).

In this paper we focus our attention to elliptic equations  with {\em piecewise constant} coefficients in $\bR^2$.
In fact, the results in \cite{MR0382836, MR2338417} imply the  $W_p^2$, $1<p<\infty$, estimate for such equations if coefficients are constants on the upper half plane and another constants on the lower half plane.
On the other hand, as a special case of the results in \cite{MR3033629}, we have the $W_p^2$, $1<p\le 2$ (or $2 \le p < \infty$), estimate for equations defined in the upper half plane with the Dirichlet (or Neumann, respectively) boundary condition if coefficients are constants on the first quadrant and another constants on the second quadrant.
In view of these results, it is then natural to ask the following question: do we have the $W^2_p$ estimate for elliptic operators with piecewise constant coefficients which are constant on each quadrant in $\bR^2$? Note that this case is not covered by any counterexamples mentioned above.

%In this paper, we
The objective of this paper is to give a negative answer to this question for any $p\in (1,\infty)$ and $p\neq 2$.
To the best of our knowledge, this result is new.
By a simple argument, our counterexamples are extended to $\bR^d$, $d \ge 3$. For divergence form equations, a similar estimate cannot be expected either due to an example by Piccinini and Spagnolo \cite{MR0361422}; see Remark \ref{rem0217_1} below.
Regarding the weak uniqueness of martingale problems, we note that Bass and Pardoux \cite{MR917679} proved well-posedness for operators in $\bR^d$ with piecewise constant coefficients.

As the main result of this paper, we give counterexamples to the statements below.
The first one is a kind of interior estimates.

\begin{statement}
							\label{stat0206_1}
For any $u \in W_p^2(\bR^2)$,
$$
\|D^2u\|_{L_p(\bR^2)} \le N \|\cL u\|_{L_p(\bR^2)} + N \|u\|_{L_p(\bR^2)},
$$
where $N$ is independent of $u$.
\end{statement}

We set
$$
\cL_{\mu} = (1-\mu) \Delta + \mu \cL,
\quad
\mu \in [0,1].
$$

To find a unique solution $u \in W_p^2(\bR^2)$ to the equation $\cL u - \lambda u = f$ in $\bR^2$, where $f \in L_p(\bR^2)$, usually one uses the method of continuity and the solvability of a simple equation, for example, the Laplace equation.
For the method of continuity, we need the following a priori estimate.
Note that Statement \ref{stat0206_2} clearly implies Statement \ref{stat0206_1}.

\begin{statement}
							\label{stat0206_2}
There exists $\lambda> 0$ such that,
for any $u \in W_p^2(\bR^2)$ and $\mu \in [0,1]$,
$$
\lambda \|u\|_{L_p(\bR^2)} + \sqrt{\lambda}\|Du\|_{L_p(\bR^2)} + \|D^2u\|_{L_p(\bR^2)} \le N \| \cL_\mu u - \lambda u\|_{L_p(\bR^2)},
$$
where $N$ is independent of $u$ and $\mu$.
\end{statement}

Here is our main result of this paper.

\begin{theorem}                           \label{thm0206_1}
\mbox{}
\begin{enumerate}
\item[(i)] For any $p \in (2, \infty)$, there exists an elliptic operator $\cL$ in non-divergence form with coefficients constant on each quadrant in $\bR^2$ such that Statement \ref{stat0206_1} does not hold.

\item[(ii)] For any $p \in (1,2)$, there exists an elliptic operator $\cL$ in non-divergence form with coefficients constant on each quadrant in $\bR^2$ such that Statement \ref{stat0206_2} does not hold.

\end{enumerate}
\end{theorem}

Theorem \ref{thm0206_1} can be used to construct a counterexample for $d \ge 3$ and $p \in (1,\infty)$, $p \ne 2$.
Indeed, we can proceed as follows when $p \in (1,2)$. If Statement \ref{stat0206_2} with $\bR^d$ in place of $\bR^2$ were satisfied by the operator
$$
\cL^d : = \cL + \Delta_{d-2} = \sum_{i,j=1}^2a^{ij}D_{ij} + \Delta_{d-2},
$$
where $a^{ij}$ are from Theorem \ref{thm0206_1} (ii), then, for functions
$$
v_n (x) := u(x_1,x_2) \varphi_n(x_3,\ldots,x_d),
$$
where $u \in W_p^2(\bR^2)$,
we would have
\begin{align*}
&\| \varphi_n\|_{L_p(\bR^{d-2})}\big(\lambda\|u\|_{L_p(\bR^2)}+\sqrt\lambda \|Du\|_{L_p(\bR^2)}+\|D^2 u\|_{L_p(\bR^2)}\big)\\
&\le N \| \varphi_n\|_{L_p(\bR^{d-2})}\| \cL_\mu u - \lambda u\|_{L_p(\bR^2)}
+ N \|\Delta_{d-2}\varphi_n\|_{L_p(\bR^{d-2})}\|u\|_{L_p(\bR^2)}.
\end{align*}
If we choose $\varphi_n(x)$ to be
$$
\varphi_n(x) = \varphi( x/n ),
\quad
\varphi \in C_0^\infty(\bR^{d-3}),
$$
then dividing both sides of the above inequality by $\| \varphi_n\|_{L_p(\bR^{d-2})}$ and letting $n \to \infty$, we would arrive at the estimate \eqref{stat0206_2} for the operator $\cL$ in $\bR^2$, which is a contradiction to Theorem \ref{thm0206_1} (ii).
Similar argument applies to the case $p\in (2,\infty)$.

\begin{remark}
It follows from Theorem \ref{thm0206_1} (ii) that, for any $p \in (1,2)$ and $\lambda \in (0,\infty)$, there exists an elliptic operator $\cL$ with coefficients constant on each quadrant in $\bR^2$ and a number $\mu_0 \in (0,1]$ such that the following statement {\em does not} hold: for any $u \in W_p^2(\bR^2)$,
\begin{equation}
                                \label{eq4.06}
\lambda\|u\|_{L_p(\bR^2)}+\sqrt \lambda \|Du\|_{L_p(\bR^2)}+\|D^2u\|_{L_p(\bR^2)} \le N \| \cL_{\mu_0} u - \lambda u\|_{L_p(\bR^2)},
\end{equation}
where $N$ is independent of $u$.
To see this, for a given $p \in (1,2)$, take the elliptic operator $\cL$ with piecewise constant coefficients from Theorem \ref{thm0206_1} (ii). Then by Theorem \ref{thm0206_1} (ii), for each positive integer $k$, one can find $\mu_k\in [0,1]$ and $u_k\in W^2_p(\bR^2)$ such that
\begin{equation}
                        \label{eq4.08}
\lambda \|u_k\|_{L_p(\bR^2)} + \sqrt{\lambda}\|Du_k\|_{L_p(\bR^2)} + \|D^2u_k\|_{L_p(\bR^2)} > k \| \cL_{\mu_k} u_k - \lambda u_k\|_{L_p(\bR^2)}.
\end{equation}
After taking a subsequence, we may assume that $\mu_k\to \mu_0$ as $k\to \infty$ for some $\mu_0\in [0,1]$.
Suppose that \eqref{eq4.06} holds for this $\mu_0$ and any $u\in W^2_p(\bR^2)$. Then for $k>2N$ sufficiently large such that $|\mu_k-\mu_0|\le 1/(2N)$, by the triangle inequality, we have
\begin{align*}
&\lambda\|u\|_{L_p(\bR^2)}+\sqrt \lambda \|Du\|_{L_p(\bR^2)}+\|D^2u\|_{L_p(\bR^2)} \\
&\,\le N \| \cL_{\mu_k} u - \lambda u\|_{L_p(\bR^2)}+\frac 1 2 \|D^2u\|_{L_p(\bR^2)},
\end{align*}
which contradicts with \eqref{eq4.08}.
\end{remark}

The remaining part of the paper is organized as follows. We give the proof of the case when $p\in (2,\infty)$, in the next section. In Section \ref{sec3}, we treat elliptic equations in divergence form. Finally, in Section \ref{sec4} we prove Theorem \ref{thm0206_1} in the case $p\in (1,2)$ by using the result in Section \ref{sec3} and a duality argument.

\section{Proof of Theorem \ref{thm0206_1} ({\rm i})}

In this section we prove Theorem \ref{thm0206_1} (i) by constructing a sequence of $\{u_n\} \subset W_p^2(\bR^2)$, $p>2$, such that the $L_p$-norms of $\cL u_n$ and $u_n$ are uniformly bounded, but those of $D^2 u_n$ are unbounded.

First we give some notation and two key, but simple, lemmas which are used in the rest of the paper.
Set
\begin{align*}
\Omega_{\omega} &= \{ (r\cos \theta, r \sin \theta) \in \bR^2,  0 < r < \infty, \theta \in (0, \omega)\},\\
\Xi &= \{ (x_1,x_2) \in \bR^2, x_1>0, x_2>0 \},\\
\bR^{2}_+&=\{x = (x_1,x_2) \in \bR^2: x_1 > 0 \}.
\end{align*}
In the sequel, we assume $p\in (1,\infty)$.

\begin{lemma}
							\label{lem0425-1}
Let $\omega \in (0, \pi)$. Then there exists a linear transformation $T$ from $\Omega_\omega$ to the first quadrant $\Xi$ of $\bR^2$.
\end{lemma}

\begin{proof}
Set $T(x) = Ax$, where $x \in \bR^2$ and
$$
A =
\left[
\begin{matrix}
1 & -\cot \omega \\
0 & 1
\end{matrix}
\right].
$$
\end{proof}

In the lemma below we use the imaginary part of a holomorphic function, which can also be found, for instance, in \cite[Page x]{Gr85} as an example illustrating the loss of smoothness at a corner point.

\begin{lemma}
							\label{lem0425-2}
For any $p > 2$, there exist a constant $\omega \in (\pi/2,\pi)$ and a function $v \in W_2^2(B_\tau \cap \Omega_\omega) \cap C^\infty_{\text{loc}}(\overline{\Omega_\omega} \setminus 0)$ such that
\begin{equation*}
%							\label{eq0425-1}
\Delta v = 0
\quad
\text{in}
\,\,\,
\Omega_\omega,
\quad
v = 0
\quad
\text{on}
\,\,\,
\partial \Omega_\omega,
\end{equation*}
and $v, Dv \in L_p(B_\tau \cap \Omega_\omega)$,
but $D^2 v \notin L_p(B_\tau \cap \Omega_\omega)$ for any $\tau \in (0,\infty)$.
\end{lemma}

\begin{proof}
For later reference in this paper, we give a proof here.
For a given $p>2$, set
$$
\omega = \frac{\pi p}{2p-2} \in (\pi/2, \pi),
\quad
v(r,\theta) = r^{\pi/\omega} \sin \left( \frac{\pi}{\omega} \theta \right).
$$
Then $\Delta v = 0$ in $\Omega_{\omega}$ and $v = 0$ on $\partial \Omega_\omega$.
Moreover, we have
$$
D v \sim r^{\pi/\omega - 1},
\quad
D^2 v \sim r^{\pi/\omega- 2}.
$$
This shows that $v, Dv \in L_p(B_\tau \cap \Omega_\omega)$, $D^2 v \notin L_p(B_\tau \cap \Omega_\omega)$, $\tau \in (0,\infty)$. Indeed,
$$
\int_{B_\tau \cap \Omega_\omega} |D^2 v|^p \, dx \, dy
\sim \int_0^\tau r^{p \left( \pi/\omega - 2 \right)+1} \, dr
= \infty
$$
since $p ( \pi/\omega - 2 )+1 = -1$.
\end{proof}

\begin{proof}[Proof of Theorem \ref{thm0206_1} (i)]
Let $\eta(r)$ be an infinitely differentiable function defined in $\bR$ such that
$$
\eta(r) = 0
\quad
\text{if}
\,\,\,
r \le 0,
\quad
\eta(r) = 1
\quad
\text{if}
\,\,\,
r \ge 1.
$$
Then set
$$
\zeta_n(x) = \eta\left( n ( |x| - 1/n) \right) \eta (3 - |x|),
\quad n =1, 2, \ldots,
$$
which has the following properties.
$$
\zeta_n(x) =
\left\{
\begin{aligned}
0
\quad
&\text{if}
\,\,\,
|x| \le 1/n
\,\,\,
\text{or}
\,\,\,
|x| \ge 3,
\\
1
\quad
&\text{if}
\,\,\,
2/n \le |x| \le 2,
\end{aligned}
\right.
$$
$$
| D \zeta_n (x) | \le N n,
\quad
| D^2 \zeta_n(x)| \le N n^2,
$$
where $N$ is a constant independent of $n$.
For a given $p>2$, take $\omega \in (\pi/2,\pi)$ and $v$ from Lemma \ref{lem0425-2}, and set
$$
v_n(x) = v(x) \zeta_n(x),
$$
which satisfies $v_n \in W_p^2(\Omega_\omega)$, $v_n = 0$ on $\partial \Omega_\omega$, and
$$
\Delta v_n = 2 (D_1 v D_1 \zeta_n + D_2 v D_2 \zeta_n) + v \Delta \zeta_n =: h_n
$$
in $\Omega_\omega$.
From direction calculations, we see that $\|v_n\|_{L_p(\Omega_\omega)}$ and $\|h_n\|_{L_p(\Omega_\omega)}$
are uniformly bounded independent of $n$, but
\begin{equation}
							\label{eq0218_1}
\|D^2 v_n\|_{L_p(\Omega_\omega)}^p \sim \ln n \to \infty
\end{equation}
as $n \to \infty$.

Now we repeat the argument in Remark 3.2 in \cite{MR3033629}.
By applying the linear transform from Lemma \ref{lem0425-1} to the equation $\Delta v_n = h_n$ in $\Omega_\omega$, we obtain constant coefficients $a^{ij}$ and
functions $u_n \in W_p^2(\Xi)$, $f_n \in L_p(\Xi)$ satisfying
$$
a^{ij} D_{ij} u_n = f_n
\quad
\text{in}
\,\,\,
\Xi,
\quad
u_n = 0
\quad
\text{on}
\,\,\,
\partial \Xi.
$$
Just for reference, here are the explicit values of the coefficients $a^{ij}$.
$$
\left[ a^{ij} \right]
= \left[\begin{array}{cc}1+ \cot^2 \omega & - \cot \omega \\- \cot \omega & 1\end{array}\right].
$$

Now we extend the equation into one defined in $\bR^2_+ = \{x \in \bR^2: x_1>0\}$ by using odd / even extensions of $u_n$,$f_n$, and the coefficients with respect $x_2$. Precisely, set
$$
\bar{u}_n(x_1,x_2) =
\left\{
\begin{aligned}
u_n(x_1,x_2),
\quad x_2 > 0,
\\
-u_n(x_1,-x_2),
\quad x_2 < 0,
\end{aligned}
\right.
$$
$$
\bar{f}_n(x_1,x_2) =
\left\{
\begin{aligned}
f_n(x_1,x_2),
\quad x_2 > 0,
\\
-f_n(x_1,-x_2),
\quad x_2 < 0,
\end{aligned}
\right.
$$
and $\bar{a}^{11} = \bar{a}^{11}$, $\bar{a}^{22} = \bar{a}^{22}$,
$$
\bar{a}^{12} = \bar{a}^{21}(x_1,x_2) =
\left\{
\begin{aligned}
a^{12},
\quad x_2 > 0,
\\
-a^{12},
\quad x_2 < 0.
\end{aligned}
\right.
$$
It then follows that $\bar{u}_n \in W_p^2(\bR^2_+)$, $\bar{f}_n \in L_p(\bR^2_+)$, and
$$
\bar{a}^{ij} D_{ij} \bar{u}_n = \bar{f}_n
\quad
\text{in}
\,\,\,
\bR^2_+,
\quad
\bar{u} = 0
\quad
\text{on}
\,\,\,
\partial \bR^2_+.
$$
Finally, we extend this equation to one defined in $\bR^2$ using similar extensions (now with respect to $x_1$) as above so that we have
$\tilde{u}_n \in W_p^2(\bR^2)$, $\tilde{f}_n \in L_p(\bR^2)$, and
$$
\tilde{a}^{ij} D_{ij} \tilde{u}_n = \tilde{f}_n
\quad
\text{in}
\,\,\,
\bR^2.
$$
However, by recalling
\eqref{eq0218_1} as well as the fact that the $L_p$-norms of $v_n$ and $h_n$
are uniformly bounded independent of $n$, and keeping track of the extensions performed to construct $\tilde{u}_n$ and $\tilde{f}_n$, we conclude that
there is no constant $N$ satisfying the inequality in Statement \ref{stat0206_1} for the sequence $\{ \tilde{u}_n\} \subset W_p^2(\bR^2)$ if $\cL = \tilde{a}^{ij}D_{ij}$.
\end{proof}

\begin{remark}
							\label{rem0218_1}
Note that in the proof above the coefficients $a^{ii}$ and $\bar{a}^{ii}$, $i = 1, 2$, are extended only evenly.
Thus the coefficients $\tilde{a}^{ii}$, $i=1,2$, are constant functions in $\bR^2$. In particular, $\tilde{a}^{22} = 1$.
\end{remark}

\begin{remark}
%							\label{rem0218_2}
Considering $\bar{u}_n$ and $\bar{f}_n$ in the proof of Theorem \ref{thm0206_1} ({\rm i}), it is clear that, for any $p \in (2, \infty)$, there exists an elliptic operator $\cL$ in non-divergence form with coefficients constant on each quadrant in $\bR^2_+$ such that the following estimate, which is a version of Statement \ref{stat0206_1} for $\bR^2_+$, {\em does not} hold: for any $u \in W_p^2(\bR^2_+)$ with $u = 0$ on $\partial \bR^2_+$,
$$
\|D^2 u\|_{L_p(\bR^2_+)}
\le N \|\cL u\|_{L_p(\bR^2_+)} + N \|u\|_{L_p(\bR^2_+)},
$$
where $N$ is independent of $u$.

On the other hand, in \cite{MR3033629} and \cite{DongKim2013}, for $p \in (1,2]$, we obtained an estimate as in Statement \ref{stat0206_2} for non-divergence type elliptic and parabolic equations in $\bR^d_+$ with the Dirichlet boundary condition when the coefficients do not have any regularity in a tangential direction.
This certainly includes the coefficients in the proof of Theorem \ref{thm0206_1} (i).
Therefore, the range of $p$ in the results of \cite{MR3033629, DongKim2013} for the Dirichlet case is sharp.
See also Remark \ref{rem0219_1} below for the Neumann case.
\end{remark}

\section{Divergence case}
                        \label{sec3}

In this section we prove a version of Theorem \ref{thm0206_1} (i) for divergence type equations with piecewise constant coefficients in $\bR^2$, which serves as an important step toward the second assertion of the main theorem (Theorem \ref{thm0206_1}).

Set
$$
\fL u = D_i ( a^{ij} D_j u )
$$
to be an operator in divergence form, where $a^{ij}$ satisfy an ellipticity condition as in \eqref{eq0217_1}.
%Also set
%$$
%\fL_{\mu} = (1-\mu) \Delta + \mu \fL, \quad \mu \in [0,1].
%$$

\begin{statement}
							\label{stat0217_1}
For any $u \in W_p^1(\bR^2)$, $g = (g_1,g_2) \in L_p(\bR^2)$, and $f \in L_p(\bR^2)$ satisfying
$$
\fL u = D_i g_i + f
$$
in $\bR^2$, we have
$$
\|Du\|_{L_p(\bR^2)} \le N \|g\|_{L_p(\bR^2)} + N \|f\|_{L_p(\bR^2)} + N \|u\|_{L_p(\bR^2)},
$$
where $N$ is independent of $u$, $g$, and $f$.
\end{statement}

Here is a version of Theorem \ref{thm0206_1} (i) for divergence type equations.

\begin{theorem}                           \label{thm0217_1}
For any $p \in (2, \infty)$, there exists an elliptic operator $\fL$ in divergence form with coefficients constant on each quadrant in $\bR^2$ such that Statement \ref{stat0217_1} does not hold.
\end{theorem}

\begin{proof}
For a given $p > 2$, take $\tilde{u}_n \in W_p^2(\bR^2)$, $\tilde{f_n} \in L_p(\bR^2)$, and the coefficients $\tilde{a}^{ij}$ from the proof of Theorem \ref{thm0206_1} (i) above.
We set $v_n := D_2 \tilde{u}_n$, $g_n := \tilde{f}_n$, and
$$
a^{11} := \tilde{a}^{11},
\quad
a^{22} := \tilde{a}^{22},\quad
a^{12} :=0,
\quad
a^{21} := \tilde{a}^{12}+ \tilde{a}^{21}.
$$
Then $v_n \in W_p^1(\bR^2)$ and, by differentiating both sides of the equation $\tilde{a}^{ij} D_{ij}\tilde{u}_n = \tilde{f}_n$ with respect to $x_2$, we get
$$
D_i (a^{ij} D_j v_n ) = D_2 g_n
\quad
\text{in}
\,\,\, \bR^2.
$$
Indeed,
\begin{align*}
&D_2\left(\tilde{a}^{11} D_{11} \tilde{u}_n\right)
+
D_2\left(\tilde{a}^{12} D_{12} \tilde{u}_n\right)
+
D_2\left(\tilde{a}^{21} D_{21} \tilde{u}_n\right)
+
D_2\left(\tilde{a}^{22} D_{22} \tilde{u}_n\right)\\
&= D_1 (a^{11} D_1 v_n)
+ D_2 \left( a^{21} D_1 v_n\right)
+ D_2 ( a^{22} D_2 v_n ),
\end{align*}
where we used the fact that $\tilde{a}^{11}$ is constant in $\bR^2$. See Remark \ref{rem0218_1}.
Then from the proof of Theorem \ref{thm0206_1} (i) above, it is clear that there is
no constant $N$ satisfying the inequality in Statement \ref{stat0217_1} for the sequence $\{ v_n\} \subset W_p^1(\bR^2)$ if $\fL = D_i (a^{ij}D_j )$.
\end{proof}

\begin{remark}
							\label{rem0218_3}
Take $\bar{u}_n$ and $\bar{f}_n$ from the proof of Theorem \ref{thm0206_1} (i) and repeat the proof of Theorem \ref{thm0217_1} with $\bar{u}_n$ and $\bar{f}_n$ in place of $\tilde{u}_n$ and $\tilde{f}_n$. In particular, $v_n := D_2 \bar{u}_n \in W_p^1(\bR^2_+)$ satisfies $v_n = 0$
on $\partial \bR^2_+$  and
$$
D_i (a^{ij} D_j v_n) = D_2 g_n
\quad
\text{in}
\,\,\,
\bR^2_+,
$$
where $g_n := \bar{f}_n$.
%See Lemma 4.3 in \cite{DongKim2013}.
Then we see that that, for any $p \in (2, \infty)$, there exists an elliptic operator $\fL$ in divergence form with coefficients constant on each quadrant in $\bR^2_+$ such that the following estimate, which is a version of Statement \ref{stat0217_1} for $\bR^2_+$, {\em does not} hold: for any $u \in W_p^1(\bR^2_+)$, $g = (g_1,g_2) \in L_p(\bR^2_+)$, and $f \in L_p(\bR^2_+)$ satisfying $u = 0$ on $\partial \bR^2_+$ and
$\fL u = D_i g_i + f$ in $\bR^2_+$, we have
$$
\|D u\|_{L_p(\bR^2_+)}
\le N \|g\|_{L_p(\bR^2_+)} + N\|f\|_{L_p(\bR^2_+)}+ N \|u\|_{L_p(\bR^2_+)},
$$
where $N$ is independent of $u$, $g$, and $f$.
\end{remark}

Regarding $W_p^1$ estimates for divergence type equations, Piccinini and Spagnolo \cite[Example 2]{MR0361422} gave a counterexample for equations with {\em symmetric} piecewise constant coefficients when $p > 2$.
To popularize this example, we present it in Remark \ref{rem0217_1} below.

\begin{remark}
							\label{rem0217_1}
Let $0 < \theta_0 < \pi/2$, $\nu = \frac{4}{\pi} \theta_0$, $K=1/\tan^2 \theta_0$, and define
$$
a(\theta) =
\left\{
\begin{aligned}
&1 \quad \text{for} \quad 0 \le \theta < \frac \pi 2, \quad \pi \le \theta < \frac 3 2 \pi,
\\
&K \quad \text{for} \quad  \frac \pi 2 \le \theta < \pi, \quad \frac 3 2 \pi \le \theta < 2 \pi,
\end{aligned}
\right.
$$
and
$$
u(x) = u(r, \theta) = r^\nu w(\theta),
$$
where
$$
w(\theta) =
\left\{
\begin{aligned}
\sin \left( \nu \left( \theta -\frac \pi 4 \right)\right)
\quad
&\text{for}
\quad
0 \le \theta < \frac \pi 2,
\\
\frac{1}{\sqrt{K}} \cos \left( \nu \left( \theta - \frac 3 4 \pi \right) \right)
\quad
&\text{for}
\quad
\frac \pi 2 \le \theta < \pi,
\\
- \sin \left( \nu \left( \theta - \frac 5 4 \pi \right) \right)
\quad
&\text{for}
\quad
\pi \le \theta < \frac 3 2 \pi,
\\
- \frac{1}{\sqrt{K}} \cos \left( \nu \left( \theta - \frac 7 4 \pi \right) \right)
\quad
&\text{for}
\quad
\frac 3 2 \pi \le \theta < 2 \pi.
\end{aligned}
\right.
$$
Then one can check that $u \in W_2^1(B_R)$ and
$$
\fL u := \sum_{i=1}^2 D_i \left( a(\theta) D_i u \right) = 0
$$
in $B_R$  for any $R > 0$.
On the other hand,
$$
D u \notin W_p^1(B_R)
$$
unless $p < \frac{2}{1-\nu} \in (2,\infty)$.
If $K \to \infty$, then $\nu \to 0$ and $\frac{2}{1-\nu} \to 2$.
Therefore, by multiplying a cutoff function we obtain a counterexample to $W_p^1$ estimates (or an estimate as in Statement \ref{stat0217_1}) for divergence type equations.
\end{remark}

\section{Proof of Theorem \ref{thm0206_1} ({\rm ii})}
                                        \label{sec4}

\begin{proof}[Proof of Theorem \ref{thm0206_1} (ii)]
We prove by contradiction. For a given $p\in (1,2)$, let $\fL=D_i(a^{ij}D_{j})$ be the divergence type operator from Theorem \ref{thm0217_1} corresponding to $q=p/(p-1)>2$.
Keep in mind that the coefficients $a^{11}$ and $a^{22}$ are constant and $a^{22} = 1$. That is,
$$
\fL=D_1(a^{11}D_1)+D_2(a^{21}D_1)+D_2^2.
$$

To get a contradiction, suppose that the non-divergence type operator
$$
\cL:=a^{11}D_1^2+a^{21}D_{12}+D_2^2
$$
satisfies Statement \ref{stat0206_2} with some $\lambda>0$.
Then by the method of continuity, for any $f\in L_p(\bR^2)$, there is a unique solution $u\in W^2_p(\bR^2)$ satisfying
$$
\cL u - \lambda u=f
\quad
\text{in}\,\,\,\bR^2
$$
and
$$
\lambda \|u\|_{L_p(\bR^2)} + \sqrt{\lambda}\|D u\|_{L_p(\bR^2)} + \|D^2 u\|_{L_p(\bR^2)}
\le N\|f\|_{L_p(\bR^2)},
$$
where $N$ is independent of $u$.
It is easily seen that $v:=D_1 u \in W_p^1(\bR^2)$ satisfies the divergence form equation
\begin{equation}
							\label{eq0218_3}
\fL^* v-\lambda v=D_1(a^{11}D_1 v)+D_1(a^{21}D_2 v)+D_2^2 v-\lambda v=D_1 f
\end{equation}
in $\bR^2$.
Furthermore, we have
\begin{equation}
                                    \label{eq3.33}
\sqrt{\lambda}\|v\|_{L_p(\bR^2)} +
\|D v\|_{L_p(\bR^2)} \le N\|f\|_{L_p(\bR^2)}.
\end{equation}

Now we take $v_n \in W_q^1(\bR^2)$ and $g_n \in L_q(\bR^2)$ from the proof of Theorem \ref{thm0217_1} corresponding to $q \in (2,\infty)$. They satisfy
\begin{equation}
                                    \label{eq3.44}
\fL v_n =D_2 g_n
\quad
\text{in}
\,\,\,
\bR^2.
\end{equation}
Then from the equations \eqref{eq0218_3}, \eqref{eq3.44}, and the estimate \eqref{eq3.33}, we get
$$
\int_{\bR^2} f D_1 v_n \, dx
= \int_{\bR^2} \left( a^{11} D_1 v_n D_1 v  + a^{21} D_1 v_n D_2 v  + D_2 v_n D_2 v  + \lambda v_n v \right) \, dx
$$
$$
= \int_{\bR^2} \left(g_n D_2 v + \lambda v_n v\right) \, dx
\le N \|f\|_{L_p(\bR^2)}
\left( \|g_n\|_{L_q(\bR^2)} + \sqrt{\lambda}\|v_n\|_{L_q(\bR^2)}\right).
$$
This implies 
\begin{equation*}
							%\label{eq0218_4}
\|D_1 v_n\|_{L_q(\bR^2)} \le N \|g_n\|_{L_q(\bR^2)}
+ N \|v_n\|_{L_q(\bR^2)},
\end{equation*}
where $N$ is independent of $n$.
This is, however, impossible from the construction of $v_n$ and $g_n$, thus a contradiction.
\end{proof}

\begin{remark}
							\label{rem0219_1}
When $p \in (1,2)$, consider a version of Statement \ref{stat0206_2} for non-divergence type equations in $\bR^2_+$ with the zero Neumann boundary condition:
there exists $\lambda > 0$ such that,
for any $\mu \in [0,1]$ and $u \in W_p^2(\bR^2_+)$ with $D_1u = 0$ on $\partial \bR^2_+$,
\begin{equation}
							\label{eq0218_2}
\lambda \|u\|_{L_p(\bR^2_+)} + \sqrt{\lambda}\|Du\|_{L_p(\bR^2_+)} + \|D^2u\|_{L_p(\bR^2_+)} \le N \| \cL_\mu u - \lambda u\|_{L_p(\bR^2_+)},
\end{equation}
where $N$ is independent of $u$ and $\mu$.
If we follow the proof of Theorem \ref{thm0206_1} (ii) with the statements in Remark \ref{rem0218_3}, then we see that, for a given $p \in (1,2)$, there is an operator $\cL$ in non-divergence form with coefficients constant on each quadrant in $\bR^2_+$ such that the estimate \eqref{eq0218_2} {\em does not} hold.

On the other hand, in \cite{MR3033629} and \cite{DongKim2013}, for $p\ge 2$  we obtained the estimate \eqref{eq0218_2} for non-divergence form elliptic and parabolic operators in $\bR^d_+$ with the Neumann boundary condition when the coefficients have no regularity assumptions in a tangential direction.
This includes the operator $\cL$ having piecewise constant coefficients.  Therefore, the range of $p$ in the results of \cite{MR3033629, DongKim2013} for the Neumann case is sharp.
\end{remark}

\bibliographystyle{plain}

\begin{thebibliography}{10}


\bibitem{MR917679}
R.~F. Bass and {\'E}.~Pardoux.
\newblock Uniqueness for diffusions with piecewise constant coefficients.
\newblock {\em Probab. Theory Related Fields}, 76(4):557--572, 1987.


\bibitem{MR1511579}
Serge Bernstein.
\newblock Sur la g\'en\'eralisation du probl\`eme de {D}irichlet.
\newblock {\em Math. Ann.}, 69(1):82--136, 1910.

\bibitem{MR0224996}
S.~Campanato.
\newblock Un risultato relativo ad equazioni ellittiche del secondo ordine di
  tipo non variazionale.
\newblock {\em Ann. Scuola Norm. Sup. Pisa (3)}, 21:701--707, 1967.

\bibitem{MR3033629}
Hongjie Dong.
\newblock On elliptic equations in a half space or in convex wedges with
  irregular coefficients.
\newblock {\em Adv. Math.}, 238:24--49, 2013.

\bibitem{DongKim2013}
Hongjie Dong and Doyoon Kim.
\newblock Parabolic equations in simple convex polytopes with time irregular
  coefficients.
\newblock {\em SIAM J. Math. Anal., to appear}.

\bibitem{MR2678983}
Hongjie Dong and N.~V. Krylov.
\newblock Second-order elliptic and parabolic equations with {$B(\mathbb R^2,
  \rm{VMO})$} coefficients.
\newblock {\em Trans. Amer. Math. Soc.}, 362(12):6477--6494, 2010.

\bibitem{Gr85}
P. Grisvard.
\newblock {\em Elliptic problems in nonsmooth domains.}
\newblock Monographs and Studies in Mathematics, 24. Pitman (Advanced Publishing Program), Boston, MA, 1985.

\bibitem{MR2240183}
Doyoon Kim.
\newblock Second order parabolic equations and weak uniqueness of diffusions
  with discontinuous coefficients.
\newblock {\em Ann. Sc. Norm. Super. Pisa Cl. Sci. (5)}, 5(1):55--76, 2006.

\bibitem{MR2276531}
Doyoon Kim.
\newblock Second order elliptic equations in {${\mathbb R}^d$} with piecewise
  continuous coefficients.
\newblock {\em Potential Anal.}, 26(2):189--212, 2007.

\bibitem{MR2338417}
Doyoon Kim and N.~V. Krylov.
\newblock Elliptic differential equations with coefficients measurable with
  respect to one variable and {VMO} with respect to the others.
\newblock {\em SIAM J. Math. Anal.}, 39(2):489--506, 2007.

\bibitem{MR0255954}
N.~V. Krylov.
\newblock Minimax type equations in the theory of elliptic and parabolic
  equations on the plane.
\newblock {\em Mat. Sb. (N.S.)}, 81 (123):3--22, 1970.

\bibitem{MR2667637}
N.~V. Krylov.
\newblock About an example of {N}. {N}. {U}ral'tseva and weak uniqueness for
  elliptic operators.
\newblock In {\em Nonlinear partial differential equations and related topics},
  volume 229 of {\em Amer. Math. Soc. Transl. Ser. 2}, pages 131--144. Amer.
  Math. Soc., Providence, RI, 2010.

\bibitem{MR0296490}
Alfredo Lorenzi.
\newblock On elliptic equations with piecewise constant coefficients.
\newblock {\em Applicable Anal.}, 2(1):79--96, 1972.

\bibitem{MR0382836}
Alfredo Lorenzi.
\newblock On elliptic equations with piecewise constant coefficients. {II}.
\newblock {\em Ann. Scuola Norm. Sup. Pisa (3)}, 26:839--870, 1972.

\bibitem{MR2260015}
Antonino Maugeri, Dian~K. Palagachev, and Lubomira~G. Softova.
\newblock {\em Elliptic and parabolic equations with discontinuous
  coefficients}, volume 109 of {\em Mathematical Research}.
\newblock Wiley-VCH Verlag Berlin GmbH, Berlin, 2000.

\bibitem{MR1612401}
Nikolai Nadirashvili.
\newblock Nonuniqueness in the martingale problem and the {D}irichlet problem
  for uniformly elliptic operators.
\newblock {\em Ann. Scuola Norm. Sup. Pisa Cl. Sci. (4)}, 24(3):537--549, 1997.

\bibitem{MR0361422}
L.~C. Piccinini and S.~Spagnolo.
\newblock On the {H}\"older continuity of solutions of second order elliptic
  equations in two variables.
\newblock {\em Ann. Scuola Norm. Sup. Pisa (3)}, 26:391--402, 1972.

\bibitem{MR0499720}
Sandro Salsa.
\newblock Un problema di {C}auchy per un operatore parabolico con coefficienti
  costanti a tratti.
\newblock {\em Matematiche (Catania)}, 31(1):126--146 (1977), 1976.

\bibitem{MR0201816}
Giorgio Talenti.
\newblock Sopra una classe di equazioni ellittiche a coefficienti misurabili.
\newblock {\em Ann. Mat. Pura Appl. (4)}, 69:285--304, 1965.

\bibitem{MR0204845}
Giorgio Talenti.
\newblock Equazioni lineari ellittiche in due variabili.
\newblock {\em Matematiche (Catania)}, 21:339--376, 1966.

\bibitem{MR0226179}
N.~N. Ural{\cprime}ceva.
\newblock The impossibility of {$W_{q}{}^{2}$} estimates for multidimensional
  elliptic equations with discontinuous coefficients.
\newblock {\em Zap. Nau\v cn. Sem. Leningrad. Otdel. Mat. Inst. Steklov.
  (LOMI)}, 5:250--254, 1967.

\end{thebibliography}

\def\cprime{$'$}\def\cprime{$'$} \def\cprime{$'$} \def\cprime{$'$}
  \def\cprime{$'$} \def\cprime{$'$}

\end{document}